\documentclass[12pt]{amsart}
\usepackage{amsmath, amsthm, amscd, amsfonts, amssymb, graphicx, color}
\usepackage{enumerate}
\usepackage[bookmarksnumbered, colorlinks, plainpages]{hyperref}
\usepackage{tikz}
\usepackage{comment}

\newtheorem{theorem}{Theorem}[section]
\newtheorem{lemma}[theorem]{Lemma}
\newtheorem{proposition}[theorem]{Proposition}

\newtheorem{corollary}[theorem]{Corollary}
\newtheorem{definition}[theorem]{Definition}

\newtheorem{remark}[theorem]{Remark}

\excludecomment{mysection}

\begin{document}
	
\title[Coexistence and orthogonality]{Coexistence of Hilbert space effects and orthogonality}

\author[A. K. Karn]{Anil Kumar Karn}

\address{School of Mathematical Sciences, National Institute of Science Education and Research Bhubaneswar, An OCC of Homi Bhabha National Institute, P.O. - Jatni, District - Khurda, Odisha - 752050, India.}

\email{\textcolor[rgb]{0.00,0.00,0.84}{anilkarn@niser.ac.in}}

	
	
\begin{abstract}
	In this paper, we show that every pair of absolutely compatible Hilbert space effects are coexistent and exhibit a partial orthogonality property. We introduce the notion of partially ortho-coexistence. We generalize absolute compatibility to obtain more examples of partially ortho-coexistent pairs and introduce the notion of generalized compatibility. In the case of $\mathbb{M}_2$, we discuss a geometric behaviour of the generalized compatibility. 
\end{abstract}

\maketitle 

\section{Introduction} 

The study of measurements is one of the main objectives of quantum mechanics. In the classical formulation, an observable is represented by a projection valued measure defined on the Borel sets of $\mathbb{R}$. However, this formulation assumes that measurements are accurate which is far from reality. 

To overcome this situation, G. Ludwig proposed an alternative formulation. In his mathematical formulation of quantum mechanics, a quantum event is represented by a self-adjoint operator on a complex Hilbert space whose spectrum lies in $[0, 1]$. Such operators are called \emph{effects} or more specifically, \emph{Hilbert space effects} and the set of all effects is called the \emph{effect algebra}. 

Let $H$ be a complex Hilbert space and let $B(H)$ denote the set of all bounded linear operators on $H$. We write 
$$B(H)_{sa} := \lbrace x \in B(H): x = x^* \rbrace$$
where $a^*$ is the adjoint operator of $a \in B(H)$ and 
$$B(H)^+ := \lbrace a \in B(H)_{sa}: \langle a \xi, \xi \rangle \ge 0 ~ \mbox{for all} ~ \xi \in H \rbrace.$$
For $x, y \in B(H)_{sa}$, we define $x \le y$ (or $y \ge x$), if $y - x \in B(H)^+$. 

The effect algebra of the Hilbert space $H$ is denoted by $E(H)$. Thus 
\begin{eqnarray*}
	E(H) &=& \lbrace a \in B(H)_{sa}: 0 \le a \le I \rbrace \\ 
	&=& \lbrace a \in B(H)^+: \Vert a \Vert \le 1 \rbrace
\end{eqnarray*} 
where $I$ is the identity operator on $H$. 

In the Ludwig's formulation of quantum mechanics, \emph{Coexistence} (of effects) is one of the important relations. A set $A$ of effects in $E(H)$ is said to be a set of \emph{coexistent effects}, if there exists a Boolean ring $\Sigma$ with an additive measure $F: \Sigma \to E(H)$ such that $A \subset F(\Sigma)$. (For the detailed discussion please refer to, for example, \cite[D.1.2.2]{GL83}.) 

Ludwig established that a pair of effects $a, b \in E(H)$ is coexistent if (and only if) there exist $x, y, z \in E(H)$ such that $a = x + y$ and $b = x + z$ with $x + y + z \in E(H)$. (See \cite[Theorem 1.2.4]{GL83}.) An equivalent formulation was observed in \cite{BS10}: $a$ and $b$ are coexistent if and only if there exist $c, d \in E(H)$ such that $c \le a \le d$, $c \le b \le d$ and $a + b = c + d$. In fact, one can choose $c = x$ and $d = x + y + z$. Note that if $a, b \in E(H)$ are coexistent then so is (each pair in) the set $\lbrace a, b, I - a, I - b \rbrace$. 

As an example, it is known that if $a, b \in E(H)$ with $a b = b a$, then $a$ and $b$ are coexistent. Conversely, if $a \in E(H)$, $p$ a projection in $B(H)$ and $a$ and $p$ are coexistent, then $a p = p a$. So a more non-trivial example would be a non-commuting pair of coexistent effects. Though such an example can be fabricated, the author could not find any mention of naturally arising non-commuting pairs of coexistent effects in the literature related to quantum mechanics. 

The notion of absolute compatibility between a pair of effects was introduced and studied in \cite{K18,JKP,K20,K22}. We observe that an absolutely compatible pair of effects is an example of a coexistent pair of effects. We know that an absolutely compatible pair of effects need not commute. In fact, an absolutely compatible pair of strict effects do not commute. It was proved in \cite{JKP} that an absolutely compatible pair of strict effects in $E(\mathbb{C}^2)$ are mixed states. 

The richness of this class motivates us to expand the scope. We introduce the notion of \emph{$x_0$-compatibility} in a pair of effects $a$ and $b$ in $E(H)$ with respect to $x_0 \in E(H)$ where $\Vert x_0 \Vert < 1$ and $x_0$ commutes with both $a$ and $b$. We see that absolute compatibility is $0$-compatibility. Thus $x_0$-compatibility is a generalization of absolute compatibility. We further observe that an $x_0$-compatible pair of effects is again coexistent. In this way, we obtain some more classes of examples of coexistent pair of effects. 

We also obtain a decomposition of an $x_0$-compatible pair of effects in terms of suitable projections. This generalizes a similar decomposition known for absolutely compatible pairs. Further, we explore a geometric behaviour of an $x_0$-compatible pair of effects in the case of $\mathbb{M}_2$. This description underscores an intrinsic property of an orthogonal (or equivalently, coexistent) pair of rank one projections in $\mathbb{M}_2$ being carried forward by a $\lambda$-compatible pair of ($\lambda$-strict) effects in $E(\mathbb{C}^2)$. (Absolutely compatible pairs of strict effects fall between the two classes.) 

Let us recall that the rank one projections in $\mathbb{M}_2$ correspond to the pure states. Whereas the $\lambda$-strict effects that engage in $\lambda$-compatibility correspond to mixed states. We know that coexistent (or equivalently, absolutely compatible) pairs of rank one projections in $\mathbb{M}_2$ are precisely orthogonal pairs. Interestingly, an $\lambda$-compatible pair of effects in $\mathbb{M}_2$ also exhibit some orthogonality albeit at a lesser scale. This observation make $\lambda$-compatible pairs of $\lambda$-strict effects in $\mathbb{M}_2$ (which are mixed states) a curious case to study further. So we are hopeful that the absolutely compatible pairs of effects and for that matter $x_0$-compatible pair of effects will find a suitable place in the theory of Experimental (Quantum) Physics.

A commuting pair of effects as well as an absolutely compatible pair of effects possess an additional property besides being coexistent pairs. To showcase this, we introduce the notion of \emph{partial ortho-coexistence}. Let $a, b \in E(H)$.  We say that $a$ and $b$ are \emph{partially ortho-coexistent}, if there exist $x, y, z \in E(H)$ such that $a = x + y$, $b = x + z$, $x + y + z \in E(H)$ and $y z = 0$.  We show that commuting pairs of effects as well $x_0$-compatible pairs of effects are partial ortho-coexistent. In particular, a pair of decision effects (that is projections in $B(H)$) is coexistent if and only if it is partially ortho-coexistent. 

We can summarize these concept as the following.
\begin{enumerate}
	\item Coexistent pair of effects: there exist $x, y, z, w \in E(H)$ such that $x + y + z + w = I$. 
	\item Partially ortho-coexistent pair of effects: there exist $x, y, z, w \in E(H)$ with $y z = 0$ such that $x + y + z + w = I$. 
	\item Absolutely compatible pair of effects: there exist $x, y, z, w \in E(H)$ with $y z = 0$ and $x w = 0$ such that $x + y + z + w = I$. 
	\item Generalized compatible pair of effects: there exist $x_0, x, y, z, w \in E(H)$ with $\Vert x_0 \Vert < 1$, $y z = 0$ and $x w = 0$ such that $x_0 + x + y + z + w = I$.
\end{enumerate}

We have organised the paper in the following way. In Section 2, we introduce the notion of partially ortho-coexistent effects and discuss some of its elementary properties. We show that a commuting pair of effects are partially ortho-coexistent. 

In Section 3, we recall the notion of an absolutely compatible pair of effect and show that it is partially ortho-coexistent. Next, we introduce the notion of an $x_0$-compatible pair of effects and study its properties on the lines of absolutely compatible pairs. We also descibe its geometric behaviour in the case of $\mathbb{M}_2$. 

In Section 4, we summarize the description of an absolutely compatible pair of effects.

\section{Partial ortho-coexistence} 

In this section, we include an additional condition to the coexistence of a pair of effects. We repeat the following notion mentioned in the introduction. 
\begin{definition}
	Let $a, b \in E(H)$. We say that $a$ and $b$ are \emph{partially ortho-coexistent}, if there exist $x, y, z \in E(H)$ such that $a = x + y$, $b = x + z$, $x + y + z \in E(H)$ and $y z = 0$. 
\end{definition} 
We recall that $x y = 0$ (that is, $x$ is \emph{(algebraically) orthogonal} to $y$), if and only if $\vert x - y \vert = x + y$. (Here $\vert w \vert := (w^* w)^{\frac 12}$ for all $w \in B(H)$.)  Thus 
$$\vert a - b \vert = \vert y - z \vert = y + z = a + b - 2 x$$ 
so that $x = \frac 12 \lbrace a + b - \vert a - b \vert \rbrace := a \dot{\wedge} b$. Again 
$$x + y + z = x + \vert a - b \vert = \frac 12 \lbrace a + b + \vert a - b \vert \rbrace := a \dot{\vee} b.$$ 
Hence we conclude that if $a$ and $b$ are partially ortho-coexistent, then $a \dot{\wedge} b, a \dot{\vee} b \in E(H)$. 

Alert readers may recall that the notions $\dot{\wedge}$ and $\dot{\vee}$ were considered in \cite{K18} to introduce the idea of \emph{absolutely ordered vector spaces} which include vector lattices as examples. It was shown in \cite{K21} that in a vector $V$, $a \dot{\wedge} b = a \wedge b$ and $a \dot{\vee} b = a \vee b$ whenever $a, b \in V$. 

Let $a, b \in B(H)_{sa}$. 
%
Then $a \dot{\wedge} b \le a \le a \dot{\vee} b$, $a \dot{\wedge} b \le b \le a \dot{\vee} b$ and $a \dot{\wedge} b + a \dot{\vee} b = a + b$. 
Further 
$$(a - (a \dot{\wedge} b)) + (b - (a \dot{\wedge} b)) = \vert a - b \vert = \vert (a - (a \dot{\wedge} b)) - (b - (a \dot{\wedge} b)) \vert$$ 
Thus $a$ and $b$ are partially ortho-coexistent, if $a \dot{\wedge} b, a \dot{\vee} b \in E(H)$. We record this conclusion in the following result.  
\begin{theorem}\label{poc}
	Let $a, b \in E(H)$. Then $a$ and $b$ are partially ortho-coexistent if and only if $a \dot{\wedge} b, a \dot{\vee} b \in E(H)$. In this case, $x, y, z \in E(H)$ are determined uniquely whenever $a = x + y$, $b = x + z$ and $y z = 0$. In fact, we have $x = a \dot{\wedge} b$, $y = (a - b)^+$, $z = (a - b)^-$ and $x + y + z = a \dot{\vee} b$. (Here $u^+ := \frac 12 (\vert u \vert + u)$ and $u^- := \frac 12 (\vert u \vert - u)$, if $u \in B(H)_{sa}$.)
\end{theorem} 
\begin{corollary}\label{pcoc} 
	Let $a, b \in E(H)$ be partially ortho-existent. Then $I - a$ and $I - b$ are also partial ortho-coexistent in $E(H)$. 
\end{corollary} 
\begin{proof}
	By Theorem \ref{poc}, we have $a \dot{\wedge} b, a \dot{\vee} b \in E(H)$ so that $0 \le a \dot{\wedge} b \le a \dot{\vee} b \le I$. Thus $0 \le I - (a \dot{\vee} b) \le I - (a \dot{\wedge} b) \le I$. Since 
	$$u - (v \dot{\wedge} w) = (u - v) \dot{\vee} (u - w)$$ 
	and 
	$$u - (v \dot{\vee} w) = (u - v) \dot{\wedge} (u - w)$$ 
	for all $u, v, w \in B(H)_{sa}$, we have 
	$$0 \le (I - a) \dot{\wedge} (I - b) \le (I - a) \dot{\vee} (I - b) \le I.$$ 
	Thus $(I - a) \dot{\wedge} (I - b), (I - a) \dot{\vee} (I - b) \in E(H)$. That is, $I - a$ and $I - b$ are also partial ortho-coexistent in $E(H)$. 
\end{proof} 
\begin{theorem}\label{poc1} 
	Let $a, b \in E(H)$. Then the following statements are equivalent:
	\begin{enumerate}
		\item $a$ and $b$ are partially ortho-coexistent;
		\item $a = a_1 + a_2$ for some $a_1, a_2 \in E(H)$ such that $a_1 \le a \dot{\wedge} b$ and $a_2 \le I - b$; 
		\item $b = b_1 + b_2$ for some $b_1, b_2 \in E(H)$ such that $b_1 \le a \dot{\wedge} b$ and $b_2 \le I - a$.
	\end{enumerate}
\end{theorem} 
\begin{proof}
	First, we assume that $a$ is partially ortho-coexistent with $b$. Then by Theorem \ref{poc}, $a \dot{\wedge} b, a \dot{\vee} b \in E(H)$. Set $a _1 := a \dot{\wedge} b$ and $a_2 := a - a_1$. As $a \dot{\wedge} b \le a$, we get that $a_1, a_2 \in E(H)$. Also 
	$$b + a_2 = a + b - (a \dot{\wedge} b) = a \dot{\vee} b \le I$$ 
	so that $a_1 \le a \dot{\wedge} b$ and $a_2 \le I - b$. 
	
	Conversely, we assume that $a = a_1 + a_2$ for some $a_1, a_2 \in E(H)$ with $a_1 \le a \dot{\wedge} b$ and $a_2 \le I - b$. Then 
	$$a - (a \dot{\wedge} b) \le a - a_1 = a_2.$$ 
	Thus 
	\begin{eqnarray*}
		I - (a \dot{\vee} b) &=& I - [a + b - (a \dot{\wedge} b)] \\ 
		&=& (I - b) - [a - (a \dot{\wedge} b)] \\ 
		&\ge& (I - b) - a_2 \ge 0.
	\end{eqnarray*} 
	Thus $a \dot{\wedge} b, a \dot{\vee} b \in E(H)$. Now, by Theorem \ref{poc}, $a$ is partially ortho-coexistent. As $\dot{\wedge}$ and $\dot{\vee}$ are symmetric, the result follows. 
\end{proof}
Returning back to $\dot{\wedge}$ an $\dot{\vee}$, we have 
\begin{proposition}
	Let $a, b \in E(H)$. Then $a \dot{\wedge} b \ge 0$ if and only if there exist $x, y, z \in E(H)$ with $y z = 0$ such that $a = x + y$ and $b = x + z$. In this case, $x + y + z = a \dot{\vee} b$.
\end{proposition} 
\begin{proof}
	First we assume that $a \dot{\wedge} b \ge 0$. Put $x = a \dot{\wedge} b$, $y = a - a \dot{\wedge} b$ and $z = b - a \dot{\wedge} b$. Then $x, y, z \in E(H)$ with $a = x + y$ and $b = x + z$. Also 
	$$y + z = a + b - 2 (a \dot{\wedge} b) = \vert a - b \vert = \vert y - z \vert$$ 
	so that $y z = 0$. 
	
	Conversely, we assume that $x, y, z \in E(H)$ with $y z = 0$ such that $a = x + y$ and $b = x + z$. Then $\vert a - b \vert = \vert y - z \vert = y + z$ so that 
	$$a \dot{\wedge} b = \frac 12 \lbrace a + b - \vert a - b \vert \rbrace = x \ge 0.$$ 
	In this case, 
	$x + y + z = a + b - (a \dot{\wedge} b) = a \dot{\vee} b$. 
\end{proof} 
\begin{proposition}
	Let $p, q \in E(H)$ be projections such that $p \dot{\wedge} q \in B(H)^+$. Then $p q = q p$. 
\end{proposition} 
\begin{proof}
	As $p \dot{\wedge} q \ge 0$, we get $\vert p - q \vert \le p + q$. Consider $p - q = x - y$ where $x, y \in B(H)^+$ and $x y = 0 (= y x)$. Then $\vert p - q \vert = \vert x - y \vert = x + y$ so that $x + y \le p + q = x - y + 2 q$. Thus, we get $0 \le y \le q$. As $p - q = x - y$, we further get $0 \le x \le p$. Let $p = x + z$. Then $z \in B(H)^+$ and $q = y + z$. Since $p$ and $q$ are projections, and since $0 \le x, z \le p$ and $0 \le y, z \le q$, we deduce that $x p = p x = x$; $z p = p z = z$; $y q = q y = y$; and $z q = q z = z$. Thus 
	$$z ( x - y) = z (p - q) = 0 = (p - q) z = (x - y) z.$$ 
	Hence $z x = z y$ and $x z = y z$. Again, as $p z = z p$, we get $x z + z^2 = z x + z^2$, or equivalently, $x z = z x$. Consequently, $y z = z y$. Thus $x + z$ commutes with $y + z$. In other words, $p q = q p$. 
\end{proof} 
\begin{proposition}\label{coco}
	Let $a, b \in E(H)$ with $a b = b a$. Then $a$ and $b$ are partial ortho-coexistent. 
\end{proposition} 
\begin{proof}
	As $a b = b a$, we have $a b \in B(H)^+$. Thus 
	$$\vert a - b \vert^2 = a^2 - 2 a b + b^2 \le a^2 - 2 a b + b^2 = (a + b)^2$$ 
	so that $\vert a - b \vert \le a + b$. Hence 
	$$a \dot{\wedge} b = \frac 12 \lbrace a + b - \vert a - b \vert \rbrace \in B(H)^+.$$ 
	Next, as $I - a, I - b \in E(H)$, replacing $a$ and $b$ by $I - a$ and $I - b$ respectively, we also have 
	$$\vert a - b \vert = \vert (I - a) - (I - b) \vert \le (I - a) + (I - B).$$ 
	Thus 
	$$a \dot{\vee} b = \frac 12 \lbrace a + b + \vert a - b \vert \rbrace \le I.$$ 
	Since $0 \le a \dot{\wedge} b \le a \dot{\vee} b \le I$, we have $a \dot{\wedge} b, a \dot{\vee} b \in E(H)$. Thus by Theorem \ref{poc}, $a$ and $b$ are partial ortho-coexistent. 
\end{proof}
\begin{remark} 
	\begin{enumerate}
		\item We recall that if $a, b \in E(H)$ are coexistent, then the decomposition $a = x + y$ and $b = x + z$ where $x, y, z \in E(H)$ with $x+ y + z \in E(H)$, in general, need not be unique. In this sense, the uniqueness of decomposition a partial ortho-coexistent pair in $E(H)$ is an additional feature. 
		\item If $a, b \in E(H)$ with $a b = b a$, then replicating the proof of Proposition \ref{coco}, we can show that each effect in $\lbrace a, I - a \rbrace$ is partial ortho-coexistent with every effect in $\lbrace b, I - b \rbrace$. 
	\end{enumerate}
	
\end{remark} 

\section{Absolute compatibility and its generalization} 

In this and the next section we shall discuss a class of partially ortho-coexistent pairs of effects. Let us recall that in \cite{K18}, the author introduced the notion of \emph{absolute compatibility}. Let $a, b \in E(H)$. We say that $a$ is \emph{absolutely compatible} with $b$, if 
$$\vert a - b \vert + \vert I - a - b \vert = I.$$ 
In this case we write $a \triangle b$. 

It follows from the symmetry in the definition that for $a, b \in E(H)$, we have $\lbrace a, I - a \rbrace \triangle \lbrace b, I - b \rbrace$ whenever $a \triangle b$. It was proved in \cite{K18} that $a \triangle (I - a)$ for some $a \in E(H)$ if and only if $a$ is a projection. It was also proved that if $a \in E(H)$ and if $p$ a projection in $B(H)$, then $a \triangle p$ if and only if $a p = p a$. For a general pair in $E(H)$ the following characterization was obtained.
\begin{theorem}\label{ac}\cite{K18}
	Let $a, b \in E(H)$. Then these statements are equivalent: 
	\begin{enumerate} 
		\item $a \triangle b$.
		\item $a \dot{\wedge} b, a \dot{\vee} b \in E(H)$ with $(a \dot{\wedge} b) (I - (a \dot{\vee} b)) = 0$. 
		\item $a \dot{\wedge} b + a \dot{\wedge} (I - b) = a$. 
		\item $a \dot{\wedge} b + (I - a) \dot{\wedge} b = b$. 
		\item $a \dot{\wedge} b + a \dot{\wedge} (I - b) + (I - a) \dot{\wedge} b + (I - a) \dot{\wedge} (I - b) = I$. 
	\end{enumerate}
\end{theorem} 
We present absolute compatibility in the language of (partially ortho-)coexistence. 
\begin{corollary}
	Let $a, b \in E(H)$. Then $a$ is absolutely compatible with $b$ if and only if there exist $x, y, z, w \in E(H)$ such that $a = x + y$, $b = x + z$, $y z = 0 = x w$ and $x + y + z + w = I$. 
\end{corollary}
\begin{proof}
	First we assume that $a$ is absolutely compatible with $b$. Put $x = a \dot{\wedge} b$, $y = a \dot{\wedge} (I - b)$, $z = (I - a) \dot{\wedge} b$ and $w = (I - a) \dot{\wedge} (I - b)$. Then by Theorem \ref{ac}, $x, y, z, w \in E(H)$ such that $a = x + y$, $b = x + z$, $y z = 0 = x w$ and $x + y + z + w = I$. 
	
	Conversely, let us assume that $x, y, z, w \in E(H)$ such that $a = x + y$, $b = x + z$, $y z = 0 = x w$ and $x + y + z + w = I$. Then 
	$$\vert a - b \vert = \vert y - z \vert = y + z = a + b - 2 x$$ 
	so that 
	$$a \dot{\wedge} b = \frac 12 \left\lbrace a + b - \vert a - b \vert \right\rbrace = x \in E(H).$$ 
	Again 
	$$a \dot{\vee} b = \frac 12 \left\lbrace a + b + \vert a - b \vert \right\rbrace = a + b - x = x + y + z = I - w \in E(H).$$ 
	Also $I - (a \dot{\vee} b) = w$ so that $(a \dot{\wedge} b) (I - (a \dot{\vee} b)) = 0$. Thus by Theorem \ref{ac} again $a \triangle b$. 
\end{proof} 
\begin{remark}
	Existence of an absolutely compatible pair in $E(H)$ leads to the following type of a `partition' of unity: there exist $x_1, x_2, y_1, y_2 \in E(H)$ with $x_1 x_2 = 0 = y_1 y_2$ such that 
	$$I = x_1 + x_2 + y_1 + y_2.$$ 
	We shall call it an \emph{absolutely compatible partition of unity}. 
\end{remark}
In particular, we deduce that an absolutely compatible pair in $E(H)$ is a partially ortho-coexistent pair. In fact, for $a, b \in E(H)$, we have $a \triangle b$ if and only if $a$ and $b$ are partially ortho-coexistent and $(a \dot{\wedge} b) (I - (a \dot{\vee} b)) = 0$. For a detailed study of absolute compatibility, we refer to \cite{JKP,K18,K20,K22}. (We have summarized the main characterization at the end of this paper.)

%
\subsection{Generalizing the absolute compatibility}

We further generalize the notion of absolute compatibility to obtain a larger class of partially ortho-coexistent pairs of effects. 
\begin{definition}
	Let $a, b \in E(H)$. We say that $a$ and $b$ are \emph{internally compatible}, if there exists $x_0 \in E(H)$ lying in the centre of von Neumann algebra $\mathcal{M}_1$ generated by $a$ and $b$ with $\Vert x_0 \Vert < 1$ such that $\vert a - b \vert + \vert I - a - b \vert = I - x_0$. In this case, we specifically say that $a$ is \emph{$x_0$-compatible} with $b$ and write $a \triangle_{x_0} b$. When $x_0 = \lambda I$ for some real number $\lambda$ with $0 \le \lambda < 1$, we say that $a$ is $\lambda$-compatible with $b$ and write $a \triangle_{\lambda} b$. 
\end{definition} 
Note that absolute compatibility coincides with $0$-compatibility. Due to the symmetry in the definition, we note that $\lbrace a, I - a \rbrace \triangle_{x_0} \lbrace b, I - b \rbrace$ whenever $a \triangle_{x_0} b$. We prove some elementary properties. 
\begin{proposition}\label{relcom}
	Let $a, b \in E(H)$ and let $\mathcal{M}_1$ be the von Neumann algebra generated by $a$ and $b$. Consider $x_0 \in B(H)^+$ lying in the centre of $\mathcal{M}_1$ with $\Vert x_0 \Vert < 1$.
	\begin{enumerate}
		\item If $a \triangle_{x_0} b$, then $a, b \in [\frac 12 x_0, I - \frac 12 x_0]$. 
		\item $a \triangle_{x_0} b$ if and only if $(I - x_0)^{-1} (a - \frac 12 x_0) \triangle (I - x_0)^{-1} (b - \frac 12 x_0)$. 
		\item If $a \triangle_{x_0} b$, then $a \dot{\wedge} b, a \dot{\vee} b \in [\frac 12 x_0, I - \frac 12 x_0]$. 
	\end{enumerate} 
\end{proposition}
\begin{proof}
	$(1)$: Let $a \triangle_{x_0} b$. Then $\vert a - b \vert + \vert I - a - b \vert = I - x_0$ so that $\pm (a - b) \pm (I - a - b) \le I - x_0$. Thus 
	\begin{eqnarray*}
		(a - b) + (I - a - b) &\le& I - x_0 \\ 
		(a - b) - (I - a - b) &\le& I - x_0 \\ 
		- (a - b) + (I - a - b) &\le& I - x_0 \\ 
		- (a - b) - (I - a - b) &\le& I - x_0.
	\end{eqnarray*} 
	Now simplifying these inequalities, we get $\frac 12 x_0 \le a \le I - \frac 12 x_0$ and $\frac 12 x_0 \le b \le I - \frac 12 x_0$. 
	
	$(2)$: Since $\Vert x _0 \Vert < 1$, we get that $I - x_0$ is invertible in $B(H)$. Also, then $(I - x_0)^{-1}$ commutes with $a$, $b$ and $x_0$. Thus $a_1 := (I - x_0)^{-1} (a - \frac 12 x_0)$ and $b_1 := (I - x_0)^{-1} (b - \frac 12 x_0)$ make sense. Now $a_1 - b_1 = (I - x_0)^{-1} (a - b)$ and $I - a_1 - b_1 = (I - x_0)^{-1} (I - a - b)$. Thus 
	$$\vert a_1 - b_1 \vert + \vert I - a_1 - b_1 \vert = (I - x_0)^{-1}(\vert a - b \vert + \vert I - a - b \vert).$$ 
	This leads to the proof of $(2)$. 
	
	$(3)$: Let $a \triangle_{x_0} b$. Then by $(2)$, we have $(I - x_0)^{-1} (a - \frac 12 x_0) \triangle (I - x_0)^{-1} (b - \frac 12 x_0)$. Thus 
	$$0 \le (I - x_0)^{-1} (a - \frac 12 x_0) \dot{\wedge} (I - x_0)^{-1} (b - \frac 12 x_0) \le I$$ 
	and 
	$$0 \le (I - x_0)^{-1} (a - \frac 12 x_0) \dot{\vee} (I - x_0)^{-1} (b - \frac 12 x_0) \le I.$$ 
	Now 
	$$(I - x_0)^{-1} (a - \frac 12 x_0) \dot{\wedge} (I - x_0)^{-1} (b - \frac 12 x_0) = (I - x_0)^{-1} (a \dot{\wedge} b - \frac 12 x_0)$$ 
	and 
	$$(I - x_0)^{-1} (a - \frac 12 x_0) \dot{\vee} (I - x_0)^{-1} (b - \frac 12 x_0) = (I - x_0)^{-1} (a \dot{\vee} b - \frac 12 x_0)$$ 
	so that $a \dot{\wedge} b, a \dot{\vee} b \in [\frac 12 x_0, I - \frac 12 x_0]$. 
\end{proof} 
\begin{remark} 
	It follows from Proposition \ref{relcom}(3) that if $a \triangle_{x_0} b$, then $a$ and $b$ are partially ortho-coexistent. 
\end{remark} 
Next, we obtain a description of internally compatible pairs of effects extending a similar characterization for absolutely compatible pairs (Theorem \ref{acc}(4)). We shall use the following result. 
\begin{lemma}\label{x_0}
	Let $\mathcal{M}$ be a von Neumann algebra and let $v \in \mathcal{M}^+$ with $\Vert v \Vert \le 1$. Then the following facts are equivalent: 
	\begin{enumerate}
		\item $v$ is strict in $\mathcal{M}$;
		\item $(1 - x) v$ and $(1 - x) (1 - v)$ are strict  in $\mathcal{M}$ whenever $x \in (\mathcal{M}^{\prime})^+$ with $\Vert x \Vert < 1$; and 
		\item $(1 - x) v$ and $(1 - x) (1 - v)$ are strict  in $\mathcal{M}$ for some $x \in (\mathcal{M}^{\prime})^+$ with $\Vert x \Vert < 1$. 
	\end{enumerate}
\end{lemma} 
\begin{proof}
	First, we assume that $v$ is strict  in $\mathcal{M}$ and let $x \in \mathcal{M}^+$ with $\Vert x \Vert < 1$. If $p$ is a projection in $\mathcal{M}$ with $p \le (1 - x) v$, then $p \le v$. As $v$ is strict  in $\mathcal{M}$, we have $p = 0$. Next, if $q$ is a projection in $\mathcal{M}$ with $(1 - x) v \le q$, then $(1 - x) v q = (1 - x) v$. As $\Vert x \vert < 1$, $1 - x$ is invertible in $\mathcal{M}$. Thus $v q = v$. Since $v$ is strict  in $\mathcal{M}$, we get that $q = 1$. Hence $(1 - x) v$ is strict  in $\mathcal{M}$. As $v$ is strict  in $\mathcal{M}$, so is $1 - v$. Thus as above $(1 - x) (1 - v)$ is also strict  in $\mathcal{M}$. Finally, we assume that $(1 - x) v$ and $(1 - x) (1 - v)$ are strict  in $\mathcal{M}$ for some $x \in \mathcal{M}^+$ with $\Vert x \Vert < 1$. Put $(1 - x) v = w$. Then $(1 - x) ( 1 - v) = 1 - x - w$. Thus $w$ and $1 - x - w$ are strict  in $\mathcal{M}$. Let $p$ be a projection in $\mathcal{M}$ with $v \le p$. Then $w \le (1 - x) p \le p$. Since $w$ is strict  in $\mathcal{M}$, we get $p = 1$. Next, let $q$ is a projection in $\mathcal{M}$ such that $q \le v$. Then $ v q = q$ so that $w q = (1 - x) q$. Thus $(1 - x - w) q = 0$. As $1 - x - w$ is strict  in $\mathcal{M}$, we conclude that $q = 0$. Hence $v$ is strict  in $\mathcal{M}$. 
\end{proof}
\begin{theorem}
	Let $a, b \in E(H)$ and let $\mathcal{M}_1$ be the von Neumann algebra generated by $a$ and $b$. Consider $x_0 \in B(H)^+$ lying in the centre of $\mathcal{M}_1$ with $\Vert x_0 \Vert < 1$. Then $a \triangle_{x_0} b$ such that $(I - x_0)^{-1} (a - \frac 12 x_0)$ and $(I - x_0)^{-1} (b - \frac 12 x_0)$ are strict  in $\mathcal{M}_1$, (that is to say that $a$ and $b$ are \emph{$x_0$-strict}  in $\mathcal{M}_1$,)  if and only if there exist $\mathbb{M}_2$-strict projections $p$ and $q$ in $\mathcal{M}_1$ and strict positive elements $w_1$ and $w_2$ in the centre of $\mathcal{M}_1$ with $w_1 + w_2 = I - x_0$ such that $a = x_0 c + w_1 p + w_2 q$ and $b = x_0 c + w_1 p + w_2 q^{\prime}$. Here $c := \frac 12 I_2$. 
\end{theorem}
\begin{proof}
	First, we assume that $a \triangle_{x_0} b$ and $(I - x_0)^{-1} (a - \frac 12 x_0)$ and $(I - x_0)^{-1} (b - \frac 12 x_0)$ are strict. As $ \triangle_{x_0} b$, by Proposition \ref{relcom}(2), we have $(I - x_0)^{-1} (a - \frac 12 x_0) \triangle (I - x_0)^{-1} (b - \frac 12 x_0)$. Note that the von Neumann algebra generated by $(I - x_0)^{-1} (a - \frac 12 x_0)$ and $(I - x_0)^{-1} (b - \frac 12 x_0)$ is again $\mathcal{M}_!$, that is, the von Neumann algebra generated by $a$ and $b$. Thus $\mathcal{M}_1$ is unitarily equivalent to $\mathbb{M}_2(\mathcal{M}_0)$ for a suitable abelian von Neumann algebra $\mathcal{M}_0$ by Theorem \ref{acc}(3). As $x_0 \in Z(\mathcal{M}_1)$, there exists $z_0 \in \mathcal{M}_0$ such that $x_0 = z_0 \otimes I_2$. Now by Theorem \ref{acc}(4), there exist $\mathbb{M}_2$-strict projections $p$ and $q$ in $\mathbb{M}_2(\mathcal{M}_0)$ and a strict positive element $y_0$ in $\mathcal{M}_0$ such that 
	$$(I - x_0)^{-1} (a - \frac 12 x_0) = ((I - y_0) \otimes I_2) p + (y_0 \otimes I_2) q$$ 
	and 
	$$(I - x_0)^{-1} (b - \frac 12 x_0) = ((I - y_0) \otimes I_2) p + (y_0 \otimes I_2) q^{\prime}.$$ 
	Put $c := \frac 12 I_2$, $v := y_0 \otimes I_2$, $w_1 := (I_2 - x_0) (I_2 - v)$ and $w_2 := (I_2 - x_0) v$. Then $a = x_0 c + w_1 p + w_2 q$ and $b = x_0 c + w_1 p + w_2 q^{\prime}$ such that $w_1$ and $w_2$ are positive elements in the centre of $\mathcal{M}_1$ with $w_1 + w_2 = I - x_0$. Also, $w_1$ and $w_2$ are strict in $\mathcal{M}_1$ by Lemma \ref{x_0}. 
	
	Conversely, we now assume that $a = x_0 c + w_1 p + w_2 q$ and $b = x_0 c + w_1 p + w_2 q^{\prime}$ for some $\mathbb{M}_2$-strict projections $p$ and $q$ in $\mathcal{M}_1$, $c := \frac 12 I$ and strict positive elements $w_1$ and $w_2$ in the centre of $\mathcal{M}_1$ with $w_1 + w_2 = I - x_0$. Then $a - b = w_2 (q - q^{\prime})$ so that $\vert a - b \vert = w_2$. Next, $I - a - b = w_1 (p^{\prime} - p)$ so that $\vert I - a - b \vert = w_1$. Thus 
	$$\vert a - b \vert + \vert I - a - b \vert = w_2 + w_1 = I - x_0$$ 
	so that $a \triangle_{x_0} b$. Further 
	$$(I - x_0)^{-1} (a - \frac 12 x_0) = (I - x_0)^{-1} w_1 p + (I - x_0)^{-1} w_2 q$$ 
	and 
	$$(I - x_0)^{-1} (b - \frac 12 x_0) = (I - x_0)^{-1} w_1 p + (I - x_0)^{-1} w_2 q^{\prime}.$$ 
	Now by Lemma \ref{x_0}, $(I - x_0)^{-1} w_1$ is strict in $\mathcal{M}_1$ and $(I - x_0)^{-1} w_1 + (I - x_0)^{-1} w_2 = I$ so that $(I - x_0)^{-1} (a - \frac 12 x_0)$ and $(I - x_0)^{-1} (b - \frac 12 x_0)$ are strict in $\mathcal{M}_1$. 
\end{proof}

\subsection{A geometric realization} 
Now, we consider internal compatibility in $E(\mathbb{C}^2)$ and obtain a geometric description with justify the adjective `internal'. Let us recall that the set of rank one projections in $\mathbb{M}_2$ can be identified with the Poincar\'{e} sphere (or Bloch sphere). Among the several identifications, we choose the real affine isomorphism 
$$\begin{bmatrix} a & \alpha \\ \bar{\alpha} & 1 - a \end{bmatrix} \leftrightarrow \left(a, \mathfrak{Re}\alpha, \mathfrak{Im}\alpha \right)$$  
between the set $\mathcal{P}_2$ of rank one projections in $\mathbb{M}_2$ and sphere $$\mathcal{B}_d = \left\lbrace (x, y, z): \left(x - \frac 12\right)^2 + y^2 + z^2 = \left(\frac 12\right)^2 \right\rbrace$$ 
in $\mathbb{R}^3$. It follows from \cite[Remark 3.4]{JKP} that if $A$ and $B$ are strict and absolutely compatible in $\mathbb{M}_2$, then $A$ and $B$ belong to the set 
$$\mathcal{S} = \lbrace X \in \mathbb{M}_2^+: 0 < \det(X) < \frac 14 ~ \textrm{and} ~ trace(X) = 1 \rbrace.$$ 
The above said identification extends to $\mathcal{S}$. Under this map, $\mathcal{S}$ corresponds to the open ball 
$$\mathcal{B}^o = \left\lbrace (x, y, z): \left(x - \frac 12\right)^2 + y^2 + z^2 < \left(\frac 12\right)^2 \right\rbrace$$ 
punctured at the centre $C_0\left( \frac 12, 0, 0 \right)$. ($C_0$ represents $\frac 12 I_2 \in \mathbb{M}_2$ which is strict but is absolutely compatible precisely with projections in $\mathbb{M}_2$ only. That is, it does not find any absolutely compatible couple in $\mathcal{S}$.) We maintain a convention that if a point $A$ lies in $\mathcal{B}_d \bigcup \mathcal{B}^o$, then the corresponding matrix in $\mathcal{P}_2 \bigcup \mathcal{S}$ is also denoted by $A$. 

We propose a definition. Let $\mathcal{P}$ be any sphere in $\mathbb{R}^3$ and assume that $A$ and $B$ are its two distinct internal points. We say that $\lbrace A, B \rbrace$ is an \emph{internal couple} with respect to $\mathcal{P}$, if the sphere described by (the line segment) $AB$ as a diameter lies inside $\mathcal{P}$. 
\begin{theorem}
	Let $\lbrace A, B \rbrace$ is an \emph{internal couple} with respect to the Poincar\'{e} sphere $\mathcal{B}_d$. Then $A \triangle_{\lambda} B$ for some suitable $\lambda$ with $0 \le \lambda < 1$. (Here, by $\lambda$-compatibility, we mean $\lambda I$-compatibility.)
\end{theorem}
\begin{proof}
	
	Let $M$ be the mid-point of the line segment $AB$. Assume that the diameter of $\mathcal{B}^o$ passing through $M$ (and $C_0$) meets the sphere at $P$ and $P'$ in such a way that $M$ lies between $P$ and $C_0$. Then the corresponding matrices are projections in $\mathcal{P}_2$ with $P' = I_2 - P$. 
	
	Let $\mathcal{P}_0$ be the sphere described by $AB$ as a diameter. Let $\mathcal{P}_0$ meet $PP'$ at $\bar{M}$ and $M_0$ in such a way that $\bar{M}$ lies between $P$ and $M$. Let $\mathcal{P}_1$ be the sphere described by $PM_0$ as a diameter. Assume that the line passing through $M_0$ and $A$ meets $\mathcal{P}_1$ at $A_1$ and the line passing through $M_0$ and $B$ meets $\mathcal{P}_1$ at $B_1$. 
	
	As $AB$ is a diameter of $\mathcal{P}_0$, we get $M_0A \perp M_0B$. Thus $M_0A_1 \perp M_0B_1$ so that $A_1B_1$ is a diameter of $\mathcal{P}_1$. Therefore, $PA_1 \perp PB_1$. Let $M_1$ be the centre of $\mathcal{P}_1$. Then $PP'$ and $A_1B_1$ intersect at $M_1$. 
	
	{\center
		\begin{tikzpicture}
			\draw (2,2) circle (3cm);
			\draw (1,2) circle (2cm);
			\draw (1.6,2) circle (1.4cm);
			\draw (-1,2) -- (5,2);
			\draw (-1,2) -- (1,4.8) -- (3,-.8) -- (-1,2);
			\draw (.36,3.9) -- (1.68,.13);
			\draw (.36,3.9) -- (3,2); 
			\draw (1.68,.13) -- (3,2);
			\draw (1.14,3.35) -- (2.08,.7);
			\draw (1,4.8) -- (5,2);
			\coordinate (A) at (1.14,3.35);
			\node[above left] at (A) {$A$};
			\coordinate (B) at (2.08,.7);
			\node[above left] at (B) {$B$}; 
			\coordinate (M_0) at (3,2);
			\node[above right] at (M_0) {$M_0$};
			\coordinate (M) at (1.61,2);
			\node[above left] at (M) {$M$};
			\coordinate (M_1) at (1.03,2);
			\node[above left] at (M_1) {$M_1$};
			\coordinate (barM) at (.4,2);
			\node[above left] at (barM) {$\bar{M}$};
			\coordinate (P) at (-1,2);
			\node[left] at (P) {$P$};
			\coordinate (P') at (5,2);
			\node[right] at (P') {$P'$};
			\coordinate (Q) at (1,4.8);
			\node[above] at (Q) {$Q$};
			\coordinate (Q') at (3,-.8);
			\coordinate (C_0) at (1.8,2);
			\node[above right] at (C_0) {$C_0$};
			\node[below] at (Q') {$Q'$};
			\coordinate (A_1) at (.36,3.9);
			\node[above left] at (A_1) {$A_1$};
			\coordinate (B_1) at (1.7,.09);
			\node[below] at (B_1) {$B_1$}; 
			\coordinate (Q') at (3,-1.5);
			\node[below left] at (3,-1.5) {Figure.1};
		\end{tikzpicture} 
	}

 	Extend $PA_1$ to meet $\mathcal{B}_d$ at $Q$ and extend $PB_1$ to meet $\mathcal{B}_d$ at $Q'$. As $PA_1 \perp PB_1$, we have $PQ \perp PQ'$. Thus $QQ'$ is a diameter of $\mathcal{B}_d$ and consequently, the matrices corresponding to the points $Q$ and $Q'$ are projections in $\mathcal{P}_2$ with $Q' = I_2 - Q$. (Figure.1.) 
 	
 	Find $\alpha, \beta, \gamma \in [0, 1]$ such that 
 	\begin{eqnarray}
 		M_0 &=& (1 - \alpha) P + \alpha P' \\ 
 		A &=& (1 - \beta) M_0 + \beta A_1 \\ 
 		A_1 &=& (1 - \gamma) P + \gamma Q.
 	\end{eqnarray} 
 	Since $\triangle M_0AB$ is similar to $\triangle M_0A_1B_1$ we get 
 	\begin{eqnarray}
 		B &=& (1 - \beta) M_0 + \beta B_1 \\ 
 		B_1 &=& (1 - \gamma) P + \gamma Q'.
 	\end{eqnarray} 
 	
 	Next, as $QP \perp QP'$, we get that $\triangle PQP'$ is similar to $\triangle PA_1M_0$. Then $\frac{PA_1}{A_1Q} = \frac{PM_0}{M_0P'}$. In other words, $\frac{\gamma}{1 - \gamma} = \frac{\alpha}{1 - \alpha}$ so that $\gamma = \alpha$. Thus, using the equations $(1)$ to $(5)$, we conclude that $A$ and $B$ can be expressed in terms of $P, P'$ and $Q$: 
 	
	\begin{eqnarray}
		A &=& (1 - \alpha) P + \alpha (1 - \beta) P' + \alpha \beta Q \\ 
		B &=& (1 - \alpha) P + \alpha (1 - \beta) P' + \alpha \beta Q'.
	\end{eqnarray} 
	If we identify $A$ and $B$ in $\mathcal{S}$ and $P$, $Q$ and $Q' := I_2 - Q$ in $\mathcal{P}_2$, the relations $(6)$ and $(7)$ still hold. Thus we get 
	$$\vert A - B \vert = \vert \alpha \beta (Q - Q') \vert = \alpha \beta I_2$$ 
	and 
	\begin{eqnarray*}
		\vert I_2 - A - B \vert &=& \vert I_2 - 2 (1 - \alpha) P - 2 \alpha (1 - \beta) P' - \alpha \beta I_2 \vert \\ 
		&=& \vert (1 - \alpha \beta) (P + P') - 2 (1 - \alpha) P - 2 \alpha (1 - \beta) P' \vert \\ 
		&=& \vert (1 - 2 \alpha + \alpha \beta) (P' - p) \vert \\ 
		&=& \vert 1 - 2 \alpha + \alpha \beta \vert I_2.
	\end{eqnarray*} 
	
	Since $\triangle M_0B_1A_1$ is congruent to $\triangle PA_1B_1$, we get that $\triangle PA_1B_1$ is similar to $\triangle PQQ'$ and consequently, $\triangle PA_1M_1$ is similar to $\triangle PQC_0$. Then 
	$$\frac{PC_0}{PM_1} = \frac{PQ}{PA_1} = \frac{PA_1 + A_1Q}{PA_1} = 1 + \frac{A_1Q}{PA_1} = 1 + \frac{1 - \gamma}{\gamma} = \frac{1}{\gamma} = \frac{1}{\alpha}.$$ 
	Since the diameter of the Poincar\'{e} sphere $\mathcal{B}_d$ is $1$, we get $PM_1 = \alpha PC_0 = \frac{\alpha}{2}$. Next, using similarity of $\triangle M_0AM$ and $\triangle\ M_0A_1M_1$, we have 
	$$\frac{M_0M_1}{M_0M} = \frac{M_0A_1}{M_0A} = \frac{M_0A + AA_1}{M_0A} = 1 + \frac{AA_1}{M_0A} = 1 + \frac{1 - \beta}{\beta} = \frac{1}{\beta}.$$ 
	Thus 
	$$\alpha = 2 PM_1 = PM_0 = PM_1 + M_1M_0 = \frac{\alpha}{2} + \frac{MM_0}{\beta}$$ 
	so that $MM_0 = \frac{\alpha \beta}{2}$. Therefore, 
	$$PM = PM_0 - MM_0 = \alpha - \frac{\alpha \beta}{2}.$$ 
	Since $PM \le \frac 12$, we get $2 \alpha - \alpha \beta \le 1$. 
	
	Thus $\vert I_2 - A - B \vert = (1 - 2 \alpha + \alpha \beta) I_2$ whence 
	$$\vert A - B \vert + \vert I_2 - A - B \vert = (1 - 2 \alpha (1 - \beta)) I_2.$$
	Put $2 \alpha (1 - \beta) = \lambda$. Then $0 \le \lambda < 1$ and $A \triangle_{\lambda} B$. 
	
\end{proof}
\begin{remark}
	We note that 
	$$\frac 12 \lambda = \alpha (1 - \beta) = PM - MM_0 = PM - M\bar{M} = P\bar{M}.$$
	Further, if we put $1 - \alpha - \alpha (1 - \beta) = \mu_1$ and $\alpha \beta = \mu_2$, then 
	$$A = \lambda C + \mu_1 P + \mu_2 Q$$ 
	and 
	$$B = \lambda C + \mu_1 P + \mu_2 Q'$$ 
	where $C = \frac 12 I_2$. This expands Theorem \ref{acc}(3) to $\lambda$-compatibility for $0 \le \lambda < 1$ when $\mathcal{M}_0 = \mathbb{C}$. 
	
	It is easy to show that if $A, B \in E(\mathbb{C}^2)$ are $\lambda$-compatible for some $0 \le 0 < 1$, then the corresponding points are an internal couple with respect to the Poincar\'{e} sphere $\mathcal{B}_d$. 
\end{remark}


\section{Appendix}  

The notion of absolute compatibility was studied and characterized in \cite{K18,JKP,K20,K22}. We summarize the results obtained in these papers to have an idea why the absolutely compatible pairs of effects form a natural class of examples for (non-commuting) pairs of coexistent effects. Though the effective part is described in part (3) of the following theorem, the other parts are included to complete the description. 
\begin{theorem}\label{acc} 
	Let $a, b \in E(H)$ with $a \triangle b$. 
	\begin{enumerate}
		\item There exist mutually orthogonal projections $p_1, p_2, p_3, p_4, p_5, p_6$ in $B(H)$ with $p_1 + p_2 + p_3 + p_4 + p_5 + p_6 = I$  such that $a$ and $b$ have the following matrix decomposition with respect to $\lbrace p_1, p_2, p_3, p_4, p_5, p_6 \rbrace$: 
		$$a = 
		\begin{bmatrix}
			p_1 & 0 & 0 & 0 & 0 & 0 \\ 0 & a_2 & 0 & 0 & 0 & 0 \\ 0 & 0 & a_{33} & a_{34} & 0 & 0 \\ 0& 0 & a_{34}^* & a_{44} & 0 & 0 \\ 0 & 0 & 0 & 0 & a_5 & 0 \\ 0 & 0 & 0 & 0 & 0 & 0
		\end{bmatrix} 
		~ \mbox{and} ~ b = 
		\begin{bmatrix}
			b_1 & 0 & 0 & 0 & 0 & 0 \\ 0 & p_2 & 0 & 0 & 0 & 0 \\ 0 & 0 & b_{33} & - a_{34} & 0 & 0 \\ 0& 0 & - a_{34}^* & b_{44} & 0 & 0 \\ 0 & 0 & 0 & 0 & 0 & 0 \\ 0 & 0 & 0 & 0 & 0 & b_6
		\end{bmatrix} 
		$$ 
		where $a_i := p_i a p_i$ and $b_i := p_i b p_i$ for $i \in \lbrace 1, 2, 5, 6 \rbrace$ and $a_{ij} = p_i a p_j$ and $b_{ij} = p_i b p_j$ for $i, j \in \lbrace 3, 4 \rbrace$. Some of them possibly may be null projections. Also. $p_3 = 0$ and $p_4 = 0$ if and only if $a b = b a$. If $a b \ne b a$, then $$\hat{a} := \begin{bmatrix} a_{33} & a_{34} \\ a_{34}^* & a_{44} \end{bmatrix} ~ \mbox{and} ~ \hat{b} := \begin{bmatrix} b_{33} & - a_{34} \\ - a_{34}^* & b_{44} \end{bmatrix}$$ are strict elements of $E((p_3 + p_4) H)$ with $\hat{a} \triangle \hat{b}$. (An effect $x \in E(H)$ is said to be \emph{strict}, if for projections $p$ and $q$ in $B(H)$ with $p \le x \le q$, we have $p = 0$ and $q = I$.) 
		
		Moreover, 
		\begin{enumerate}
			\item $a_{34} a_{34}^* = (p_3 - a_{33}) (p_3 - b_{33}) = (p_3 - b_{33}) (p_3 - a_{33})$; 
			\item $a_{34}^* a_{34} = a_{44} b_{44} = b_{44} a_{44}$; and 
			\item $a_{34} = a_{33} a_{34} + a_{34} a_{44} = b_{33} b_{34} + b_{34} b_{44}$.
		\end{enumerate} 
		\item If in addition, $a$ and $b$ are strict, then $a b \ne b a$ and in $(1)$, we have $p_1, p_2, p_5$ and $p_6$ are null projections. Moreover, there exists a projection $p$ in $B(H)$ for which $p H$ is isometrically isomorphic to $(I - p) H$; a pair of strict elements $a_0, b_0 \in E(p h)$ for which $a_0 b_0 = b_0 a_0$, $a_0^2 + b_0^2 \le p$ and $a_0^2 + b_0^2$ strict in $E(p H)$; and a unitary $u: H \to p H \oplus p H$ such that 
		$$a = u^* \begin{bmatrix} a_0^2 & a_0 b_0 \\ a_0 b_0 & p - a_0^2 \end{bmatrix} u ~ \mbox{and} ~ b = u^* \begin{bmatrix} b_0^2 & - a_0 b_0 \\ - a_0 b_0 & p - b_0^2 \end{bmatrix} u.$$
		Further, if $\mathcal{M}$ is the von Neumann algebra generated by $a$ and $b$ and if $\mathcal{M}_0$ is the (abelian) von Neumann algebra generated by $a_0$ and $b_0$, then $\mathcal{M}$ is isometrically $*$-isomorphic to $\mathbb{M}_2(\mathcal{M}_0)$. Therefore, the absolute compatibility of a pair of strict effects in $E(H)$ is in effect that of a pair of strict elements in $\mathbb{M}_2(\mathcal{M}_0)$ for some abelian von Neumann algebra $\mathcal{M}_0$. 
		\item\label Consider a pair of strict elements $A, B \in \mathbb{M}_2(\mathcal{M}_0)$ for some abelian von Neumann algebra $\mathcal{M}_0$. Then there exist a unitary $U \in \mathbb{M}_2(\mathcal{M}_0)$ and strict elements $p_1, p_2, x \in \mathcal{M}_0$ with $p_1^2 + p_2^2 = I$ such that 
		$$U^* A U = ((I - x) \otimes I_2) P_0 + (x \otimes I_2) P$$ 
		and 
		$$U^* B U = ((I - x) \otimes I_2) P_0 + (x \otimes I_2) P'.$$ 
		Here $I$ is the identity element of $\mathcal{M}_0$, $P_0 := \begin{bmatrix} 0 & 0 \\ 0 & I \end{bmatrix}$, $P := \begin{bmatrix} p_1^2 & p_1 p_2 \\ p_1 p_2 & p_2^2 \end{bmatrix}$, $I_2 := \begin{bmatrix} I & 0 \\ 0 & I \end{bmatrix}$, $P' := I_2 - P$ and $z \otimes I_2 := \begin{bmatrix} z & 0 \\ 0 & z \end{bmatrix}$. (Note that the set $\lbrace z \otimes I_2: z \in \mathcal{M}_0 \rbrace$ is the centre of $\mathbb{M}_2(\mathcal{M}_0)$.) We recall (\cite[Definition 2.2]{K22}) that a projection like $P$ above is called an $\mathbb{M}_2$-strict projection in $\mathbb{M}_2(\mathcal{M}_0)$.
	\end{enumerate} 
\end{theorem}

\thanks{{\bf Acknowledgements:} The author acknowledge with heartfelt gratitude a fruitful discussion with B. V. Rajarama Bhat that helped the author to enrich some of the results in this paper. The author declares that no data was used while preparing this paper and that there is no conflict of interest arising out of it.


\begin{thebibliography}{100}

\bibitem{BS10} P. Busch and H.-J. Schmidt, \textit{Coexistence of qubit effects}, Quantum Inf. Process, 9 (2010), 143-169.

\bibitem{JKP} N. K. Jana, A. K. Karn and A. M. Peralta, Absolutely compatible pairs in a von Neumann algebra, Electronic Journal of Linear Algebra, 35(2019), 599-618.


\bibitem{K18} A. K. Karn, \textit{Algebraic orthogonality and commuting projections in operator algebras} Acta. Sci. Math. (Szeged), \textbf{84} (2018), 323--353.

\bibitem{K20} A. K. Karn, \textit{Absolutely compatible pair of elements in a von Neumann algebra-II}, Revista de la Real Academia de Ciencias Exactas, Fisicas y Naturales. Serie A. Matem\'{a}ticas (RCSM), 114(3) (2020), Article 153, 7 pages.

\bibitem{K21} A. K. Karn, \textit{Orthogonality: an antidote to Kadison's anti-lattice theorem}, Positivity and its Applications, Trends in Mathematics, Birkhauser, Switzerland, (2021), 217-227. 

\bibitem{K22} A. K. Karn, \textit{Absolute compatibility and Poincar\'{e} sphere}, Ann. Funct. Anal., 13 (2022), Article 39, 16 pages.

\bibitem{GL83} G. Ludwig, \textit{Foundations of Quantum Mechanics} Vol. I, (Translated from German by C. A. Hein), Springer Verlag, New York (1983).

\end{thebibliography}
\end{document}